\documentclass[12pt]{amsart}%
\usepackage{graphicx}
\usepackage[small,nohug,heads=littlevee]{diagrams}
\usepackage{tikz}
\usetikzlibrary{matrix,arrows,decorations.pathmorphing}
\usepackage{amscd}
\usepackage{geometry}
\usepackage{amsmath}
\usepackage{amsfonts}
\usepackage{amssymb}%
\usepackage{enumerate}
\usepackage{extpfeil}
\usepackage{hyperref}
\providecommand{\U}[1]{\protect\rule{.1in}{.1in}}
\newtheorem{theorem}{Theorem}[section]
\theoremstyle{plain}

\newtheorem{corollary}[theorem]{Corollary}

\newtheorem{lemma}[theorem]{Lemma}
\newtheorem{question}{Question}
\newtheorem{proposition}[theorem]{Proposition}
\theoremstyle{definition}
\newtheorem{definition}{Definition}[section]

\newtheorem{remark}{Remark}

\numberwithin{equation}{section}

\usepackage [autostyle, english = american]{csquotes}
\MakeOuterQuote{"}

\makeatletter
\def\oversortoftilde#1{\mathop{\vbox{\m@th\ialign{##\crcr\noalign{\kern3\p@}%
      \sortoftildefill\crcr\noalign{\kern3\p@\nointerlineskip}%
      $\hfil\displaystyle{#1}\hfil$\crcr}}}\limits}

\def\sortoftildefill{$\m@th \setbox\z@\hbox{$\braceld$}%
  \braceld\leaders\vrule \@height\ht\z@ \@depth\z@\hfill\braceru$}

\makeatother

\begin{document}
\title[Characterization of pseudo-collarable manifolds with boundary]{Characterization of pseudo-collarable manifolds with boundary}
\author{Shijie Gu}
\address{Department of Mathematics\\
County College of Morris, Randolph, NJ, 07869}
\email{sgu@ccm.edu}
\thanks{}
\date{}

\begin{abstract}
In this paper we obtain a complete characterization of pseudo-collarable $n$-manifolds for $n\geq 6$. This extends earlier work by Guilbault and Tinsley to allow for manifolds with noncompact boundary. In the same way that their work can be viewed as an extension of Siebenmann's dissertation that can be applied to manifolds with non-stable fundamental group at infinity, our main theorem can also be viewed as an extension of the recent Gu-Guilbault characterization of completable $n$-manifolds in a manner that is applicable to manifolds whose fundamental group at infinity is not peripherally stable. 

\end{abstract}
\maketitle

In 1965, Siebenmann's PhD thesis \cite{Sie65} provided necessary and sufficient conditions for an open manifold $M^m$ of dimension at least 6 to contain an open collar neighborhood of infinity, i.e., a manifold neighborhood of infinity $N$ such that $N \approx \partial N \times [0,1)$. His collaring theorem can be easily extended to cases when $M^m$ is noncompact but with compact boundary. However, the situation becomes much subtler if $\partial M^m$ is noncompact. Instead of "collaring", the term "completion" serves as an appropriate analog. An $m$-manifold $M^{m}$ with (possibly empty) boundary is \emph{completable} if there exists a compact manifold $\widehat{M}^{m}$ and a compactum $C\subseteq\partial\widehat{M}^{m}$ such that $\widehat{M}^{m}\backslash C$ is homeomorphic to $M^{m}$. In this case $\widehat{M}^{m}$ is called a \emph{(manifold) completion }of $M^{m}$. After Siebenmann did some initial work in such topic, O'Brien [O'B83] characterized completable $m$-manifolds $(m>5)$ in the case where $M^m$ and $\partial M^m$ are both 1-ended. But, in general, a completable manifold with (noncompact) boundary may have uncountably many non-isolated ends. For example, one can take any favoriate compact manifold with boundary and remove a Cantor set from its boundary. In a very recent work, Guilbault and the author provided a complete characterization for high-dimensional manifolds (with boundary).

\begin{theorem}\cite{GG18}
\emph{[Manifold Completion Theorem]}\label{Th: Completion Theorem} An
$m$-manifold $M^{m}$ ($m\geq 6$) is completable if and only if

\begin{enumerate}
\item \label{Char1}$M^{m}$ is inward tame,

\item \label{Char2}$M^{m}$ is peripherally $\pi_{1}$-stable at infinity,

\item \label{Char3}$\sigma_{\infty}(M^{m})\in\underleftarrow{\lim}\left\{
\widetilde{K}_{0}(\pi_{1}(N))\mid N\text{ a clean neighborhood of
infinity}\right\}  $ is zero, and

\item \label{Char4}$\tau_{\infty}\left(  M^{m}\right)  \in\underleftarrow{\lim
}^{1}\left\{  \operatorname*{Wh}(\pi_{1}(N))\mid N\text{ a clean neighborhood
of infinity}\right\}  $ is zero.\bigskip
\end{enumerate}
\end{theorem}

Although Condition (\ref{Char2}) is necessary in order for such a completion to exist, such condition is too rigid to characterize many exotic examples related to current research trends in topology and geometric group theory. For instance, the exotic universal covering spaces produced by Mike Davis in \cite{Dav83} are not collarable (because Condition (\ref{Char2}) fails) yet their ends exhibit some nice geometric structure. Other examples such as (open) manifolds that satisfy Conditions (\ref{Char1}), (\ref{Char3}) and (\ref{Char4}) but not Condition (\ref{Char2}) can be found in \cite[Thm.1.3]{GT03}. Define a manifold neighborhood of infinity $N$ in a manifold $M^m$ to be a \emph{homotopy collar} provided $\operatorname{Fr}N \hookrightarrow N$ is a homotopy equivalence. A \emph{pseudo-collar} is a homotopy collar which contains arbitrarily small homotopy collar neighborhoods of infinity. A manifold is \emph{pseudo-collarable} if it contains a pseudo-collar neighborhood of infinity. When $M^m$ is an open manifold (or more generally, a manifold with compact boundary), Guilbault \cite{Gui00} initiated a program to produce a generalization of Siebenmann's collaring theorem. The idea of pseudo-collars and a detailed motivation for the definition are nicely exposited in \cite{Gui00}. Through a series of papers \cite{Gui00, GT03, GT06}, a complete characterzation for pseudo-collarable manifolds with compact boundary was provided. 

\begin{theorem}\cite{GT06}\label{Thm: pseudo compact boundary}
An $m$-manifold $M^m$ ($m\geq 6$) with compact boundary is pseudo-collarable iff each of the
following conditions holds:

\begin{enumerate}[(i)]

\item $M^{m}$ is inward tame

\item \label{Con ii} $M^{m}$ is perfectly $\pi_1$-semistable at infinity,

\item $\sigma_{\infty}(M^{m})\in\underleftarrow{\lim}\left\{
\widetilde{K}_{0}(\pi_{1}(N))\mid N\text{ a clean neighborhood of
infinity}\right\}  $ is zero. \bigskip
\end{enumerate}
\end{theorem}

Just as Theorem \ref{Th: Completion Theorem} is a natural generalization of Siebenmann's dissertation to manifolds with noncompact boundaries, it is natural to extend the study of pseudo-collarability to manifolds with noncompact boundaries. Moreover, since all completable manifolds are pseudo-collarable (a key step in the proof of Theorem \ref{Th: Completion Theorem}), a more general study of pseudo-collarability also generalizes Theorem \ref{Th: Completion Theorem} in the same way that Theorem \ref{Thm: pseudo compact boundary} generalized \cite{Sie65}. Our main result is the following characterization theorem. 

\begin{theorem}
[Pseudo-collarability characterization theorem]\label{Th: Characterization Theorem} An $m$-manifold $M^m$ ($m\geq 6$) is pseudo-collarable iff each of the
following conditions holds:

\begin{enumerate}[(a)]

\item \label{condition a} $M^{m}$ is inward tame

\item \label{condition b} $M^{m}$ is peripherally perfectly $\pi_1$-semistable at infinity,

\item \label{condition c} $\sigma_{\infty}(M^{m})\in\underleftarrow{\lim}\left\{
\widetilde{K}_{0}(\pi_{1}(N))\mid N\text{ a clean neighborhood of
infinity}\right\}  $ is zero. \bigskip
\end{enumerate}
\end{theorem}
\begin{remark}
It is worth noting that Condition (\ref{condition b}) of Theorem \ref{Th: Characterization Theorem} is strictly weaker than Condition (\ref{Char2}) of Theorem \ref{Th: Completion Theorem}. That is due to the fact that (perfect) $\pi_1$-semistability may not guarantee pro-monomorphism.
For instance, Davis's exotic universal covering spaces constructed in \cite{Dav83} satifies Condition (\ref{condition b}) but not Condition (\ref{Char2}). Furthermore, Condition (\ref{condition b}) reduces to Condition (\ref{Con ii}) of Theorem \ref{Thm: pseudo compact boundary} when boundary $\partial M^m$ is compact. 
\end{remark}

The strategy of our proof is heavily relying on techniques and results developed by several substantial and technical papers \cite{Sie65}, \cite{Gui00, GT03, GT06} and \cite{GG18}. For a full understanding, the readers should be familiar with the Pseudo-collarability Characterization Theorem in \cite{GT06} and the Manifold Completion Theorem in \cite{GG18}. We will not reprove all these results, but our goal is to take shortcuts afforded by both papers, hence, provide an efficient proof of Theorem \ref{Th: Characterization Theorem}.

About the organization of this paper: \S \ref{Section: Terminology} contains basic definitions and notation. \S \ref{Section: Peripheral stability}
provides background materials and technical set-up about neighborhoods of infinity, ends, peripheral perfect semistability condition, etc. \S \ref{Section: finite domination and inward tameness} and \S \ref{Section: Finiteness obstruction} give illustrations for Conditions (\ref{condition a}) and (\ref{condition c}) respectively. \S \ref{Section: Concatenation} sets forth some 
 lemmas. In \S \ref{Section: Proof of Characterization Th: necessity} and \S \ref{Section: Proof of Characerization Theorem: sufficiency}, we prove Theorem \ref{Th: Characterization Theorem}. In the final section of this paper, we discuss some related open questions.

\section{Conventions, notation, and terminology\label{Section: Terminology}}

For convenience, all manifolds are assumed to be piecewise-linear (PL). Equivalent results in the smooth and topological categories may be otained in the usual ways. For instance, some technical issues in smooth category requiring "rounding off corners" or "smoothing corners" have been nicely exposited in \cite{Sie65} and \cite{O'B83}. Unless stated
otherwise, an $m$-manifold $M^{m}$ is permitted to have a boundary, denoted
$\partial M^{m}$. We denote the \emph{manifold interior} by
$\operatorname*{int}M^{m}$. For $A\subseteq M^{m}$, the \emph{point-set
interior} will be denoted $\operatorname*{Int}_{M^{m}}A$ and the
\emph{frontier} by $\operatorname{Fr}_{M^{m}}A$. A \emph{closed manifold} is a compact
boundaryless manifold, while an \emph{open manifold} is a non-compact
boundaryless manifold.

A $q$-dimensional submanifold $Q^{q}\subseteq M^{m}$ is
\emph{properly embedded} if it is a closed subset of $M^{m}$ and $Q^{q}%
\cap\partial M^{m}=\partial Q^{q}$; it is \emph{locally flat} if each
$p\in\operatorname*{int}Q^{q}$ has a neighborhood pair homeomorphic to
$\left(
\mathbb{R}
^{m},%
\mathbb{R}
^{q}\right)  $ and each $p\in\partial Q^{q}$ has a neighborhood pair
homeomorphic to $\left(
\mathbb{R}
_{+}^{m},%
\mathbb{R}
_{+}^{q}\right)  $. By this definition, the only properly embedded codimension
0 submanifolds of $M^{m}$ are unions of its connected components; a more
useful variety of codimension 0 submanifolds are the following: a codimension
0 submanifold $Q^{m}\subseteq M^{m}$ is \emph{clean} if it is a closed subset
of $M^{m}$ and $\operatorname{Fr}_{M}Q^{m}$ is a properly embedded locally flat
(hence, bicollared) $\left(  m-1\right)  $-submanifold of $M^{m}$. In that
case, $\overline{M^{m}\backslash Q^{m}}$ is also clean, and $\operatorname{Fr}_{M}Q^{m}$ is a clean 
codimension 0 submanifold of both $\partial Q^{m}$ and $\partial(\overline{M^{m}\backslash Q^{m}})$.

When the dimension of a manifold or submanifold is clear, we will often omit
the superscript; for example, denoting a clean codimension 0 submanifold
simply by $Q$. Similarly, when the ambient space is clear, we denote
(point-set) interiors and frontiers by $\operatorname*{Int}A$ and $\operatorname{Fr} A$.

For any codimension 0 clean submanifold $Q\subseteq M^{m}$, let $\partial
_{M}Q$ denote $Q\cap\partial M^{m}$ and $\operatorname*{int}_{M}%
Q=Q\cap\operatorname*{int}M^{m}$; alternatively, $\partial_{M}Q=\partial
Q\backslash\operatorname{int}(\operatorname{Fr} Q)$ and $\operatorname*{int}_{M}Q=Q\backslash\partial M^{m}$.
Note that $\operatorname*{int}_{M}Q$ is a $m$-manifold and $\partial\left(
\operatorname*{int}_{M}Q\right)  =\operatorname*{int}\left( \operatorname{Fr} Q\right)
$.

\section{Ends, $\operatorname*{pro}$-$\pi_{1}$, and the peripherally perfectly semistable
condition\label{Section: Peripheral stability}}

Most of definitions and terminologies in this section are based on ones developed in \cite{GG18}. We give a brief review in terms of the topology of ends of manifolds and inverse sequences of groups.

\subsection{Neighborhoods of infinity, partial neighborhoods of infinity, and ends.} Let $M^{m}$ be a connected manifold. A \emph{clean neighborhood of infinity}
in $M^{m}$ is a clean codimension 0 submanifold $N\subseteq M^{m}$ for which
$\overline{M^{m}\backslash N}$ is compact. Equivalently, a clean neighborhood
of infinity is a set of the form $\overline{M^{m}\backslash C}$ where $C$ is a
compact clean codimension 0 submanifold of $M^{m}$. A \emph{clean compact exhaustion} of
$M^{m}$ is a sequence $\left\{  C_{i}\right\}  _{i=1}^{\infty}$ of clean
compact connected codimension 0 submanifolds with $C_{i}\subseteq
\operatorname*{Int}_{M^{m}}C_{i+1}$ and $\cup C_{i}=M^{m}$. By letting
$N_{i}=\overline{M^{m}\backslash C_{i}}$, we obtain the corresponding
\emph{cofinal sequence of clean neighborhoods of infinity}. Each such $N_{i}$
has finitely many components $\left\{  N_{i}^{j}\right\}  _{j=1}^{k_{i}}$. By
enlarging $C_{i}$ to include all of the compact components of $N_{i}$, we can
arrange that each $N_{i}^{j}$ is noncompact; then, by drilling out regular
neighborhoods of arcs connecting the various components of each $\operatorname{Fr}
_{M^{m}}N_{i}^{j}$ (thereby further enlarginging $C_{i}$), we can
arrange that each $\operatorname{Fr}_{M^{m}}N_{i}^{j}$ is connected. An $N_{i}$ with
these latter two properties is called a $0$-neighborhood of infinity. For
convenience, most constructions in this paper will begin with a clean compact
exhaustion of $M^{m}$ with a corresponding cofinal sequence of clean
0-neighborhoods of infinity.

Assuming the above arrangement, an \emph{end }of $M^{m}$ is determined by a
nested sequence of components $\varepsilon=\left(  N_{i}^{k_{i}}\right)
_{i=1}^{\infty}$ of the $N_{i}$; each component is called a \emph{neighborhood
of }$\varepsilon$. In $\S 3.3$, we discuss components $\{N^j\}$ of a neighborhood of infinity
$N$ without reference to a specific end of $M^m$. In that situation, we will refer to the
$N^j$ as \emph{partial neighborhoods of infinity} for $M^m$ (\emph{partial $0$-neighborhoods} if $N$ is a
0-neighborhood of infinity). Clearly every noncompact clean connected codimension
0 submanifold of $M^m$ with compact frontier is a partial neighborhood of infinity with respect to an appropriately chosen compact $C$; if its frontier is connected it is a partial 0-neighborhood of infinity.

\subsection{The fundamental group of an end} For each end $\varepsilon$, we will define the \emph{fundamental group at}
$\varepsilon$; this is done using inverse sequences. Two inverse sequences of
groups and homomorphisms $A_{0}\overset{\alpha_{1}}{\longleftarrow}%
A_{1}\overset{\alpha_{2}}{\longleftarrow}A_{3}\overset{\alpha_{3}%
}{\longleftarrow}\cdots$ and $B_{0}\overset{\beta_{1}}{\longleftarrow}%
B_{1}\overset{_{\beta_{2}}}{\longleftarrow}B_{3}\overset{\beta_{3}%
}{\longleftarrow}\cdots$ are \emph{pro-isomorphic} if they contain
subsequences that fit into a commutative diagram of the form
\begin{equation}
\begin{diagram} A_{i_{0}} & & \lTo^{\lambda_{i_{0}+1,i_{1}}} & & A_{i_{1}} & & \lTo^{\lambda_{i_{1}+1,i_{2}}} & & A_{i_{2}} & & \lTo^{\lambda_{i_{2}+1,i_{3}}}& & A_{i_{3}}& \cdots\\ & \luTo & & \ldTo & & \luTo & & \ldTo & & \luTo & & \ldTo &\\ & & B_{j_{0}} & & \lTo^{\mu_{j_{0}+1,j_{1}}} & & B_{j_{1}} & & \lTo^{\mu_{j_{1}+1,j_{2}}}& & B_{j_{2}} & & \lTo^{\mu_{j_{2}+1,j_{3}}} & & \cdots \end{diagram} \label{basic ladder diagram}%
\end{equation}
An inverse sequence is \emph{stable }if it is pro-isomorphic to a constant
sequence $C\overset{\operatorname*{id}}{\longleftarrow}%
C\overset{\operatorname*{id}}{\longleftarrow}C\overset{\operatorname*{id}%
}{\longleftarrow}\cdots$. Clearly, an inverse sequence is pro-isomorphic to
each of its subsequences; it is stable if and only if it contains a
subsequence for which the images stabilize in the following manner
\begin{equation}
\begin{diagram} A_{0}& & \lTo^{{\lambda}_{1}} & & A_{1} & & \lTo^{{\lambda}_{2}} & & A_{2} & & \lTo^{{\lambda}_{3}} & & A_{3} &\cdots\\ & \luTo & & \ldTo & & \luTo & & \ldTo & & \luTo & & \ldTo & \\ & & \operatorname{Im}\left( \lambda_{1}\right) & & \lTo^{\cong} & & \operatorname{Im}\left( \lambda _{2}\right) & &\lTo^{\cong} & & \operatorname{Im}\left( \lambda_{3}\right) & & \lTo^{\cong} & &\cdots & \\ \end{diagram} \label{Standard stability ladder}%
\end{equation}
where all unlabeled homomorphisms are restrictions or inclusions.

Given an end $\varepsilon=\left(  N_{i}^{k_{i}}\right)  _{i=1}^{\infty}$,
choose a ray $r:[1,\infty)\rightarrow M^{m}$ such that $r\left(
[i,\infty)\right)  \subseteq N_{i}^{k_{i}}$ for each integer $i>0$ and form
the inverse sequence
\begin{equation}
\pi_{1}\left(  N_{1}^{k_{1}},r\left(  1\right)  \right)  \overset{\lambda
_{2}}{\longleftarrow}\pi_{1}\left(  N_{2}^{k_{2}},r\left(  2\right)  \right)
\overset{\lambda_{3}}{\longleftarrow}\pi_{1}\left(  N_{3}^{k_{3}},r\left(
3\right)  \right)  \overset{\lambda_{4}}{\longleftarrow}\cdots
\label{sequence: pro-pi1}%
\end{equation}
where each $\lambda_{i}$ is an inclusion induced homomorphism composed with
the change-of-basepoint isomorphism induced by the path $\left.  r\right\vert
_{\left[  i-1,i\right]  }$. We refer to $r$ as the \emph{base ray} and the
sequence (\ref{sequence: pro-pi1}) as a representative of the
\textquotedblleft fundamental group at $\varepsilon$ based at $r$%
\textquotedblright\ ---denoted $\operatorname*{pro}$-$\pi_{1}\left(
\varepsilon,r\right)  $. We say \emph{the fundamental
group at }$\varepsilon$\emph{ is stable} if (\ref{sequence: pro-pi1}) is a
stable sequence. A nontrivial (but standard) observation is that both semistability and stability of $\varepsilon$ do
not depend on the base ray (or the system of neighborhoods if infinity used to
define it). See \cite{Gui16} or \cite{Geo08}.

If $\{H_i,\mu_i\}$ can be chosen so that each $\mu_i$ is an epimorphism,
we say that our inverse sequence is \emph{semistable} (or \emph{Mittag-Leffler}, or \emph{pro-epimorphic}).
In this case, it can be arranged that the restriction maps in the bottom row of (\ref{basic ladder diagram}) are
epimorphisms. Similarly, if $\{H_i,\mu_i\}$ can be chosen so that each $\mu_i$ is a monomorphism,
we say that our inverse sequence is \emph{pro-monomorphic}; it can then be arranged that the
restriction maps in the bottom row of (\ref{basic ladder diagram}) are monomorphisms. It is easy to see that an
inverse sequence that is semistable and pro-monomorphic is stable.

Recall that a \emph{commutator} element of a group $H$ is an element of the form $x^{-1}y^{-1}xy$
where $x,y \in H$; and the \emph{commutator subgroup} of $H$; denoted $[H,H]$ or $H^{(1)}$, is the subgroup generated by all of its commutators. The group $H$ is \emph{perfect} if $H = [H,H]$. An inverse sequence of groups is \emph{perfectly semistable} if it is pro-isomorphic to an inverse
sequence. 

\begin{equation}
G_0 \xtwoheadleftarrow{\lambda_1} G_1 \xtwoheadleftarrow{\lambda_2}G_2 \xtwoheadleftarrow{\lambda_3}\cdots 
\end{equation}
of finitely generated groups and surjections where each $\ker(\lambda_i)$ perfect. The
following shows that inverse sequences of this type behave well under passage
to subsequences.

\begin{lemma}
\label{lemma: subsequence preserves perfect semistable} A composition of surjective group homomorphisms, each having
perfect kernels, has perfect kernel. Thus, if an inverse sequence of surjective
group homomorphisms has the property that the kernel of each bonding map
is perfect, then each of its subsequences also has this property.
\end{lemma}

\begin{proof}
See \cite[Lemma 1]{Gui00}.

\end{proof}

\subsection{Relative connectedness, relatively perfectly semistability, and the peripheral perfect semistability condition.}

Let $Q$ be a manifold and $A\subseteq\partial Q$. We say that $Q$ is\emph{
}$A$-\emph{connected at infinity} if $Q$
contains arbitrarily small neighborhoods of infinity $V$ for which $A\cup V$
is connected.

\begin{lemma}\cite[Lemma 4.1]{GG18}
\label{Lemma: relA connected versus Q-A 1-ended}Let $Q$ be a noncompact
manifold and $A$ a clean codimension 0 submanifold of $\partial Q$. Then $Q$
is $A$-connected at infinity if and only
if $Q\backslash A$ is 1-ended.
\end{lemma}

If $A\subseteq\partial Q$ and $Q$ is $A$-connected at infinity: let $\left\{  V_{i}\right\}  $ be a cofinal sequence
of clean neighborhoods of infinity for which each $A\cup V_{i}$ is connected;
choose a ray $r:[1,\infty)\rightarrow\operatorname*{Int}Q$ such that $r\left(
[i,\infty)\right)  \subseteq V_{i}$ for each $i>0$; and form the inverse
sequence%
\begin{equation}
\pi_{1}\left(  A\cup V_{1},r\left(  1\right)  \right)  \overset{\mu
_{2}}{\longleftarrow}\pi_{1}\left(  A\cup V_{2},r\left(  2\right)  \right)
\overset{\mu_{3}}{\longleftarrow}\pi_{1}\left(  A\cup V_{3},r\left(  3\right)
\right)  \overset{\mu_{4}}{\longleftarrow}\cdots
\label{sequence: rel A pro-pi1}%
\end{equation}
where bonding homomorphisms are obtained as in (\ref{sequence: pro-pi1}). We
say $Q$ is $A$-\emph{perfectly $\pi_1$-semistable at infinity } (resp. $A$-$\pi_{1}%
$\emph{-stable at infinity}) if (\ref{sequence: rel A pro-pi1}) is perfectly semistable (resp. stable).
Independence of this property from the choices of $\left\{  V_{i}\right\}  $
and $r$ follows from the traditional theory of ends by applying Lemmas
\ref{Lemma: relA connected versus Q-A 1-ended} and
\ref{Lemma: relA pro-pi1 versus Q-A pro-pi1}. Because each boundary component of a manifold with boundary is collared, the following lemma is true because "throwing away" part of the boundary will preserve the homotopy type of the original manifold.

\begin{lemma}\cite[Lemma 4.2]{GG18}
\label{Lemma: relA pro-pi1 versus Q-A pro-pi1}Let $Q$ be a noncompact manifold
and $A$ a clean codimension 0 submanifold of $\partial Q$ for which $Q$ is
$A$-connected at infinity. Then, for any
cofinal sequence of clean neighborhoods of infinity $\left\{  V_{i}\right\}  $
and ray $r:[1,\infty)\rightarrow Q$ as described above, the sequence
(\ref{sequence: rel A pro-pi1}) is pro-isomorphic to any sequence representing
$\operatorname*{pro}$-$\pi_{1}\left(  Q\backslash A,r\right)  $.
\end{lemma}

\begin{remark}
In the above discussion, we allow for the possibility that $A=\varnothing$. In that case, $A$-connected at infinity reduces to 1-endedness and $A$-perfectly $\pi_1$-semistable (resp. $A$-$\pi_{1}$-stability) to ordinary perfectly semistable (resp.  $\pi_{1}$-stability) at that end.
\end{remark}

We can now formulate one of the key definitions of this paper.

\begin{definition} Let $M^m$ be an manifold and $\varepsilon$ be an end of $M^m$
\begin{enumerate}
\item \label{condition1} $M^m$ is \emph{peripherally locally connected at infinity} if it contains arbitrarily small
$0$-neighborhoods of infinity $N$ with the property that each component $N^j$ is $\partial_M N^j$-connected at infinity.
\item $M^m$ \label{condition2} is \emph{peripherally locally connected at} $\varepsilon$ if $\varepsilon$ has arbitrarily small $0$-neighbor-hoods $P$ that are $\partial_M P$-connected at infinity.
\end{enumerate} 
\end{definition}

An $N$ with the property described in Condition (\ref{condition1}) will be called a \emph{strong $0$-neigh-
borhood of infinity} for $M^m$, and a $P$ with the property described in Condition (\ref{condition2})
will be called a \emph{strong $0$-neighborhood of $\varepsilon$}. More generally, any connected partial
0-neighborhood of infinity $Q$ that is $\partial_M Q$-connected at infinity will be called a \emph{strong
partial $0$-neighborhood of infinity}.

\begin{lemma}\cite[Lemma 4.4]{GG18}
$M^m$ is peripherally locally connected at infinity iff $M^m$ is peripherally
locally connected at each of its ends.
\end{lemma}

In the next section, one will see that \textbf{every} inward tame manifold $M^{m}$ is peripherally locally connected at infinity. As a consequence, that condition plays less prominent role than the next definition.

\begin{definition}Let $M^m$ be a manifold and $\varepsilon$ an end of $M^m$.
\begin{enumerate}
\item $M^m$ is \emph{peripherally perfectly $\pi_1$-semistable at infinity} if it contains arbitrarily small strong
0-neighborhoods of infinity $N$ with the property that each component $N^j$ is $\partial_M N^j$-perfectly $\pi_1$-semistable at infinity.

\item $M^m$ is \emph{peripherally perfectly $\pi_1$-semistable at $\varepsilon$} if $\varepsilon$ has arbitrarily small strong $0$-neighborhoods $P$ that are $\partial_M P$-perfectly $\pi_1$-semistable at infinity.
\end{enumerate}
 \end{definition}

If $M^m$ contains arbitrarily small 0-neighborhoods of infinity $N$ with
the property that each component $N^j$ is $\partial_M N^j$-perfectly semistable at infinity, then those components provide arbitrarily small neighborhoods of the ends satisfying the necessary perfectly semistable condition. Thus, it's easy to see peripheral perfect $\pi_1$-semistability at infinity implies peripheral perfect $\pi_1$-semistability at each end.

\section{Finite domination and inward
tameness\label{Section: finite domination and inward tameness}}

A topological space $P$ is \emph{finitely dominated} if there exists a finite
polyhedron $K$ and maps $u:P\rightarrow K$ and $d:K\rightarrow P$ such that
$d\circ u\simeq\operatorname*{id}_{P}$. If choices can be made so that $d\circ
u\simeq\operatorname*{id}_{P}$ and $u\circ d\simeq\operatorname*{id}_{K}$,
i.e, $P\simeq K$, we say that $P$ \emph{has finite homotopy type}. For
convenience we will restrict our our attention to cases where $P$ is a locally
finite polyhedron---a class that contains the (PL) manifolds and submanifolds
(and certain subspaces of these) under consideration in this paper.

A locally finite polyhedron $P$ is \emph{inward tame} if it contains
arbitrarily small polyhedral neighborhoods of infinity that are finitely
dominated. Equivalently, $P$ contains a cofinal sequence $\left\{
N_{i}\right\}  $ of closed polyhedral neighborhoods of infinity each admitting
a \textquotedblleft taming homotopy\textquotedblright\ $H:N_{i}\times\left[
0,1\right]  \rightarrow N_{i}$ that pulls $N_{i}$ into a compact subset of
itself. By an application of the Homotopy Extension Property, we may require
taming homotopies to be fixed on $\operatorname{Fr} N_{i}$. From there, it is easy to see
that, in an inward tame polyhedron, \emph{every} closed neighborhood of
infinity admits a taming homotopy.\footnote{A discussion of the
\textquotedblleft inward tame\textquotedblright\ terminology can be found in \cite{Gui16}.}

\begin{lemma}\cite[Lemma 5.3]{GG18}
\label{Lemma: inward tameness of deleted manifolds}Let $M^{m}$ be a manifold
and $A$ a clean codimension 0 submanifold of $\partial M^{m}$. If $M^{m}$ is
inward tame then so is $M^{m}\backslash A$.
\end{lemma}

It is easy to see that a finitely dominated space $P$ has finitely generated
homology. The following result is vital to this paper.

\begin{proposition}\cite[Prop.5.4]{GG18}
\label{Prop: finite ends}
If a noncompact connected manifold $M^m$ and its boundary each
have finitely generated homology, then $M^m$ has finitely many ends. More specifically,
the number of ends of $M^m$ is bounded above by $\dim H_{m-1}(M^m, \partial M^m;$ $ \mathbb{Z}_2) + 1$.
\end{proposition}

\begin{corollary}\cite[Cor.5.5]{GG18}
\label{Cor: inward tame implies locally connected}
If $M^m$ is inward tame, then $M^m$ is peripherally locally connected at infinity.
\end{corollary}

\section{Finite homotopy type and the $\sigma_{\infty}$%
-obstruction\label{Section: Finiteness obstruction}}

Finitely generated projective left $\Lambda$-modules $P$ and $Q$ are
\emph{stably equivalent} if there exist finitely generated free $\Lambda
$-modules $F_{1\text{ }}$ and $F_{2}$ such that $P\oplus F_{1}\cong Q\oplus
F_{2}$. Under the operation of direct sum, the stable equivalence classes of
finitely generated projective modules form a group $\widetilde{K}_{0}\left(
\Lambda\right)  $, the \emph{reduced projective class group} of $\Lambda$. In
\cite{Wal65}, Wall asssociated to each path connected finitely dominated space
$P$ a well-defined $\sigma\left(  P\right)  \in\widetilde{K}_{0}\left(
\mathbb{Z}
\lbrack\pi_{1}\left(  P\right)  ]\right)  $ which is trivial if and only if
$P$ has finite homotopy type.\footnote{Here $%
\mathbb{Z}
\lbrack\pi_{1}\left(  P\right)  ]$ denotes the integral group ring
corresponding to the group $\pi_{1}\left(  P\right)  $. In the literature,
$\widetilde{K}_{0}\left(
\mathbb{Z}
\lbrack G]\right)  $ is sometimes abbreviated to $\widetilde{K}_{0}\left(
G\right)  $.} As one of his necessary and sufficient conditions for
completability of a 1-ended inward tame open manifold $M^{m}$ ($m>5$) with
stable $\operatorname*{pro}$-$\pi_{1}$, Siebenmann defined the \emph{end
obstruction} $\sigma_{\infty}\left(  M^{m}\right)  $ to be (up to sign) the
finiteness obstruction $\sigma\left(  N\right)  $ of an arbitrary clean
neighborhood of infinity $N$ whose fundamental group \textquotedblleft
matches\textquotedblright\ $\operatorname*{pro}$-$\pi_{1}\left(
\varepsilon\left(  M^{m}\right)  \right)  $.

In cases where $M^{m}$ is multi-ended or has non-stable $\operatorname*{pro}%
$-$\pi_{1}$ (or both), a more general definition of $\sigma_{\infty}\left(
M^{m}\right)  $, introduced in \cite{CS76}, is required. Here we inherit the definition discussed in \cite{GG18}. For inward tame finitely dominated locally finite polyhedron $P$ (or more
generally locally compact ANR), let $\left\{ N_{i}\right\}$ \footnote{Each $N_i$ can be non-path-connected. However, inward tameness assures that each $N_i$ has finitely many components --- each finitely dominated. See \cite{GG18} for the details about defining the functor $\tilde{K}_0$ and the finiteness obstruction in this situation.} be a nested
cofinal sequence of closed polyhedral neighborhoods of infinity and define%
\[
\sigma_{\infty}\left(  P\right)  =\left(  \sigma\left(  N_{1}\right)
,\sigma\left(  N_{2}\right)  ,\sigma\left(  N_{3}\right)  ,\cdots\right)
\in\underleftarrow{\lim}\left\{  \widetilde{K}_{0}[%
\mathbb{Z}
\lbrack\pi_{1}(N_{j})]\right\}
\]
The bonding maps of the target inverse sequence%
\[
\widetilde{K}_{0}[%
\mathbb{Z}
\lbrack\pi_{1}(N_{1})]\leftarrow\widetilde{K}_{0}[%
\mathbb{Z}
\lbrack\pi_{1}(N_{2})]\leftarrow\widetilde{K}_{0}[%
\mathbb{Z}
\lbrack\pi_{1}(N_{3})]\leftarrow\cdots
\]
are induced by inclusion maps, with the Sum Theorem for finiteness obtructions
\cite[Th.6.5]{Sie65} assuring consistency. Clearly, $\sigma_{\infty}\left(
P\right)  $ vanishes if and only if each $N_{i}$ has finite homotopy type;
by another application of the Sum Theorem, this happens if and only if
\emph{every} closed polyhedral neighborhood of infinity has finite homotopy type.

We close this section by quoting a result from \cite[Lemma 6.1]{GG18}, which plays a key role in proving Theorem \ref{Th: Characterization Theorem}.

\begin{lemma}\label{Lemma wall obstruction M - A}
Let $M^m$ be a manifold and $A$ a clean codimension $0$ submanifold of
$\partial M^m$. If $M^m$ is inward tame and $\sigma_{\infty}(M^m)$ vanishes, then $M^m \backslash A$ is inward tame
and $\sigma_{\infty} (M^m \backslash A)$ also vanishes.
\end{lemma}

\section{Concatenation of  one-sided h-precobordisms \label{Section: Concatenation}}
Recall that an (absolute) cobordism is a triple $(W,A,B)$, where $W$ is a manifold with boundary and $A$ and $B$ are disjoint manifolds without boundary for which $A \cup B = \partial W$. The triple $(W,A,B)$ is a \emph{relative cobordism} if $A$ and $B$ are disjoint codimension 0 clean submanifolds of $\partial W$. In that case, there is an associated absolute cobordism $(V, \partial A, \partial B)$ where $V = \partial W\backslash (\operatorname{int}A \cup \operatorname{int}B)$. We view absolute cobordisms as special cases of relative cobordisms where $V = \emptyset$. A relative cobordism $(W,A,B)$ is a \emph{relative one-sided h-cobordism} if one of the pairs of inclusions $(A,\partial A)\hookrightarrow (W,V)$ or $(B,\partial B)\hookrightarrow (W,V)$ is a homotopy equivalence. A relative cobordism is \emph{nice} if it is absolute or if $(V, \partial A, \partial B) \approx (\partial A \times [0, 1] , \partial A \times {0} , \partial A\times {1})$. Readers are referred to \cite{RS82} for more details of such topic.

\begin{remark}\label{Remark: nice cobordism}
A situation similar to a nice relative cobordism occurs when $\partial W =
A\cup B'$, where $A$ and $B'$ are codimension 0 clean submanifolds of $\partial W$ with a common
nonempty boundary $\partial A = \partial B'$. We call such cobordism a \emph{precobordism}. By choosing a clean codimension 0 submanifold $B \subseteq B'$ with the property that $B'\backslash \operatorname{int}B \approx \partial B \times [0, 1]$ we arrive at a nice relative
cobordism $(W,A,B)$. When this procedure is applied, we will refer to $(W,A,B)$ as
a corresponding \emph{nice relative cobordism}. For notational consistency, we will always
adjust the term $B'$ on the far right of the triple $(W,A,B')$, leaving $A$ alone. A precobordism is a 
\emph{one-sided h-precobordism} if one of the pairs of inclusions $A \hookrightarrow W$ or $B' \hookrightarrow W$ is a homotopy equivalence.

\end{remark}

The role played by one-sided $h$-precobordisms in the study of pseudo-collars is illustrated by the following easy proposition.

\begin{proposition}
\label{Prop: one-sided h cobordism}
Let $W_i$ be a disjoint union of finitely many relative one-sided $h$-cobordisms $\bigsqcup_{j} (W_i^j,A_i^j,{B_{i}^{j}})$ with $A_i^j \hookrightarrow W_i^j$ a homotopy equivalence. Let $\bigsqcup_{j}A_i^j$ and $\bigsqcup_{j}{B_{i}^{j}}$ be $A_i$ and $B_i$ respectively. Suppose for each $i\geq 1$, there is a homeomorphism $h_i: B_i \to A_{i+1}$ identifying a clean codimension 0 submanifold $B_{i}^{j} \subset B_i$ with a clean codimension 0 submanifold $A_{i+1}^{j} \subset A_{i+1}$. Then the adjunction space 
$$N = W_1 \cup_{h_1} W_2 \cup_{h_2} W_3 \cup_{h_3}\cdots $$ 
is a pseudo-collar. Conversely, every pseudo-collar may be expressed as a countable union of relative one-sided $h$-cobordisms in this manner.
\end{proposition}

\begin{proof}
For the forward implication, the definition of relative one-sided $h$-cobordism implies that $\operatorname{Fr}N = A_1 \hookrightarrow W_1 \cup_{h_1}\cdots \cup_{h_{k-1}}W_k$ is a homotopy equivalence for any finite $k$. Then a direct limit argument shows that $\operatorname{Fr}N \hookrightarrow N$ is a homotopy equivalence. Hence, $N$ is a homotopy collar. To see that $N_i$ is a pseudo-collar, we apply the same argument to the subset $N_{i} = W_{i+1} \cup_{h_{i+1}} W_{i+2} \cup_{h_{i+2}} W_{i+3} \cup_{h_{i+3}}\cdots $.

For the converse, assume $N$ is a pseudo-collar. Choose a homotopy collar $N_1 \subset \operatorname{Int}N$ and let $W_1 = N \backslash \operatorname{Int}N_1$. Then $\operatorname{Fr}N \hookrightarrow W_1$ is a homotopy equivalence. So, 
$(W_1,\operatorname{Fr}N, \operatorname{Fr} N_1\cup \partial_N W_1)$ is a one-sided $h$-precobordism. Denote a component of $N_1$ by $N_1^j$. Let $N_2'$ be the disjoint union of homotopy collars in $N_1^j$ and $W_2^j = N_1^j \backslash \operatorname{Int}N_2'$. Since $\operatorname{Fr}N_1^j \hookrightarrow W_2^j$ is a homotopy equivalence, each $(W_2^j,\operatorname{Fr}N_1^j, \operatorname{Fr} N_2'\cup \partial_{N_1^j} W_2^j)$ is a  one-sided $h$-precobordism. Repeating the procedure concludes the argument. See Figure \ref{fig 1}.

\begin{figure}[h!]
        \centering
       \includegraphics[ width=10cm, height=8.5cm]{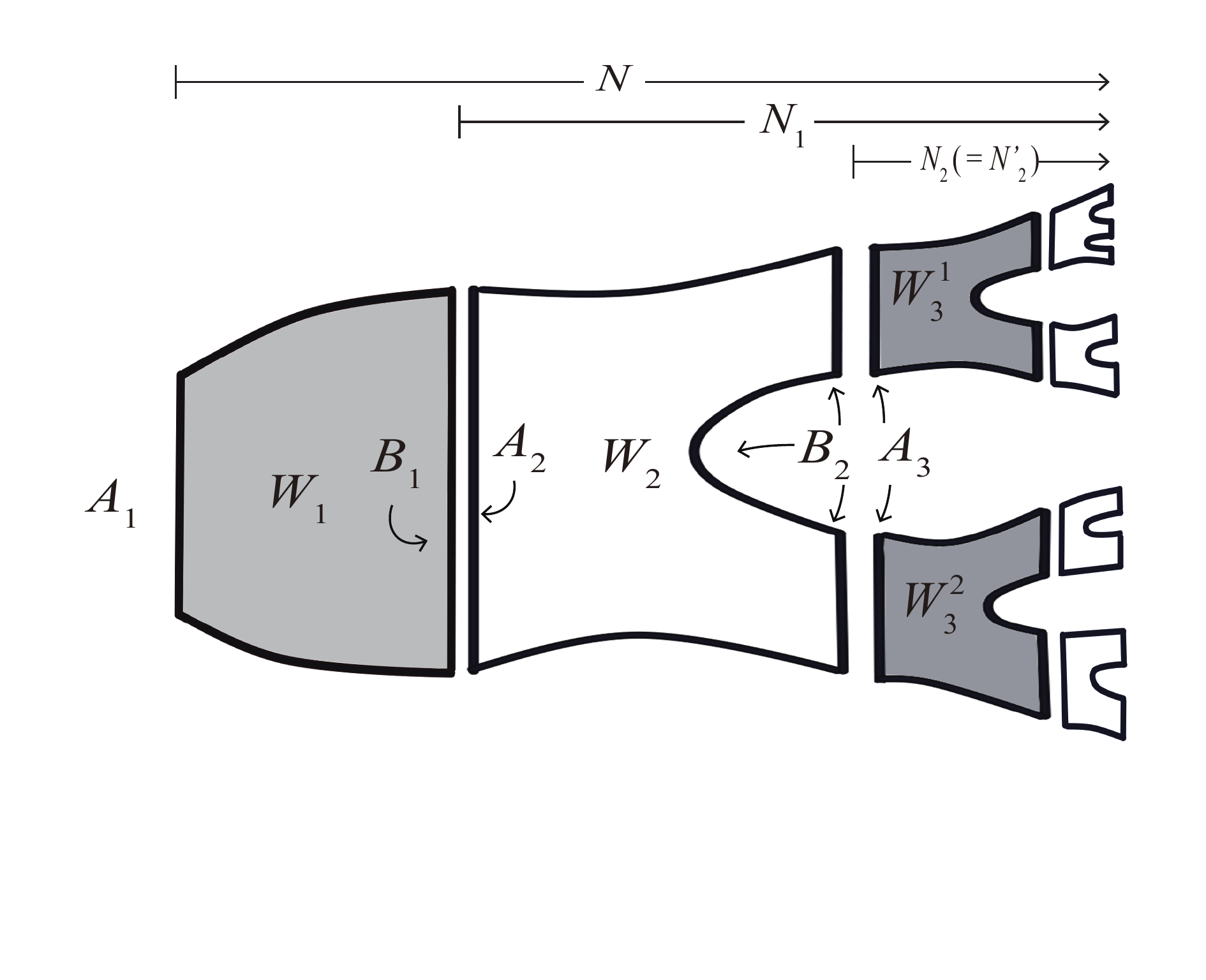}
       \vspace{-2em}
       \caption{A concatenation of relative one-sided $h$-cobordisms.}
        \label{fig 1}
\end{figure}

By the cleanliness of $\operatorname{Fr}N$ and $\operatorname{Fr}N_i$'s together with the adjustment described in Remark \ref{Remark: nice cobordism}, one can re-define one-sided $h$-precobordisms
$$(W_1,\operatorname{Fr}N, \operatorname{Fr} N_1\cup \partial_N W_1),(W_2^j,\operatorname{Fr}N_1^j, \operatorname{Fr} N_2'\cup \partial_{N_1^j} W_2^j), \dots$$ 
as nice relative cobordisms
$$(W_1,\operatorname{Fr}N, B_1), (W_2^j,\operatorname{Fr}N_1^j, B_2^j), \dots$$
where $B_1,B_2^j,\dots$ are clean codimension 0 submanifold $\subseteq \operatorname{Fr} N_1\cup \partial_N W_1$, $\operatorname{Fr} N_2'\cup \partial_{N_1^j} W_2^j,\dots,$ respectively with the property that $(\operatorname{Fr} N_1\cup \partial_N W_1)\backslash \operatorname{int}B_1 \approx \partial B_1\times [0,1]$, $(\operatorname{Fr} N_2'\cup \partial_{N_1^j} W_2^j)\backslash \operatorname{int}B_2^j \approx \partial B_2^j \times [0,1],\dots $.

Then it's easy to see that those nice relative cobordisms are relative one-sided $h$-cobordisms.
\end{proof}

The following lemma proved by duality and standard covering space theory is crucial in this paper.
\begin{lemma}
\label{perfect kernel}
Let $(W,A,B')$ be a one-sided $h$-precobordism with $A \hookrightarrow W$ a homotopy equivalence. Then the inclusion induced map
$$i_\#: \pi_1(B')\to \pi_1(W)$$
is surjective and has perfect kernel.
\end{lemma}

\begin{proof}
The proof is similar to the argument of Theorem 2.5 in \cite{GT03}. Let $p:\tilde{W}\to W$ be the universal covering projection, $\tilde{A}=p^{-1}(A)$ and $\hat{B'}=p^{-1}(B')$. By generalized Poincar\'{e} duality for non-compact manifolds \cite[Thm.3.35, P. 245]{Hat02},
$$H_k(\tilde{W},\hat{B'};\mathbb{Z}) \cong H_{c}^{n-k}(\tilde{W},\tilde{A};\mathbb{Z}),$$
where cohomology is with compact supports. Since $\tilde{A} \hookrightarrow \tilde{W}$ is a proper homotopy equivalence, all of these relative cohomology groups vanish, so $H_k(\tilde{W},\hat{B'};\mathbb{Z})=0$ for all $k$. It follows that $H_1(\tilde{W},\hat{B'};\mathbb{Z})$ vanishes. Then by considering the long exact sequence for $(\tilde{W},\hat{B'})$, we have $H_0(\hat{B'};\mathbb{Z})=\mathbb{Z}$. Thus, $\hat{B'}$ is connected. By covering space theory, the components of $\hat{B'}$ are 1-1 corresponding to the cosets of $i_\#(\pi_1(\hat{B}'))$ in $\pi_1(W)$. So, $i_\#$ is surjective. To see the kernel of $i_\#$ is perfect, we consider the long exact sequence for $(\tilde{W},\hat{B'})$ again. Using $H_2(\tilde{W},\hat{B'};\mathbb{Z})=0$ together with the simple connectivity of $\tilde{W}$, $H_1(\hat{B'};\mathbb{Z})$ vanishes. Hence, $\pi_1(\hat{B'})$ is perfect. By covering space theory, $\pi_1(\hat{B'})\cong \ker i_\#$ is perfect.
\end{proof}

The following well-known lemma plays an important role in the proof of Theorem \ref{Th: Characterization Theorem}. The proof follows easily from the Seifert-van Kampen Theorem.

\begin{lemma}
\label{amalgamation}
Let $X$ be a connected CW complex and $Y \subset X$ a connected subcomplex. Let $Y'$ be the resulting space obtained by attaching $2$-cells to $Y$ along loops $\{l_i\}$ in $Y$. Then $\pi_1(Y')\cong \pi_1(Y)/N$, where $N$ is the normal closure in $\pi_1(Y)$ of $\{\operatorname{incl}_{\#}(\pi_1(l_i))\}$. Let $X' = X \cup Y'$. Suppose $i_\# :\pi_1(Y) \to \pi_1(X)$ is the inclusion induced map. Then $\pi_1(X') \cong \pi_1(X)/N'$, where $N'$ is the normal closure in $\pi_1(X)$ of $i_\#(N)$. Thus, if $N$ is perfect, so is $N'$ (since the image of a perfect group is perfect and the normal closure of a perfect group is perfect.)
\end{lemma}

\begin{lemma}\label{Lemma: perfect kernel}
Let $P$ be a compact $(n-1)$-manifold with boundary and $\{A_i\}$ a finite collection of pairwise disjoint compact codimension $0$ clean (and connected) submanifolds of $P$. Let $\{(W_i,A_i,B_i')\}$ be a collection of one-sided $h$-precobordisms with $A_i \hookrightarrow W_i$ a homotopy equivalence.  Assume each $W_i$ intersects $P$ along $A_i$. Let $R = P \cup (\cup_i W_i)$ and $Q = (P \backslash (\cup_i A_i)) \cup (\cup_i B_i')$. Then $\pi_1(Q) \to \pi_1(R) \cong \pi_1(P)$ is surjective and has perfect kernel.
\end{lemma}
\begin{proof}
We begin with $Q = (P \backslash (\cup_i A_i)) \cup (\cup_i B_i')$. Choose a finite collection of arcs in $P$ that connect up the $A_i$. By adding tubular neighborhoods of these arcs, we get a clean connected codimension 0 submanifold $A$ of $P$. Attaching $W_1$ along $B_1'$. See Figure \ref{perfect_kernel}.

\begin{figure}[h!]
        \centering
       \includegraphics[ width=10cm, height=8cm]{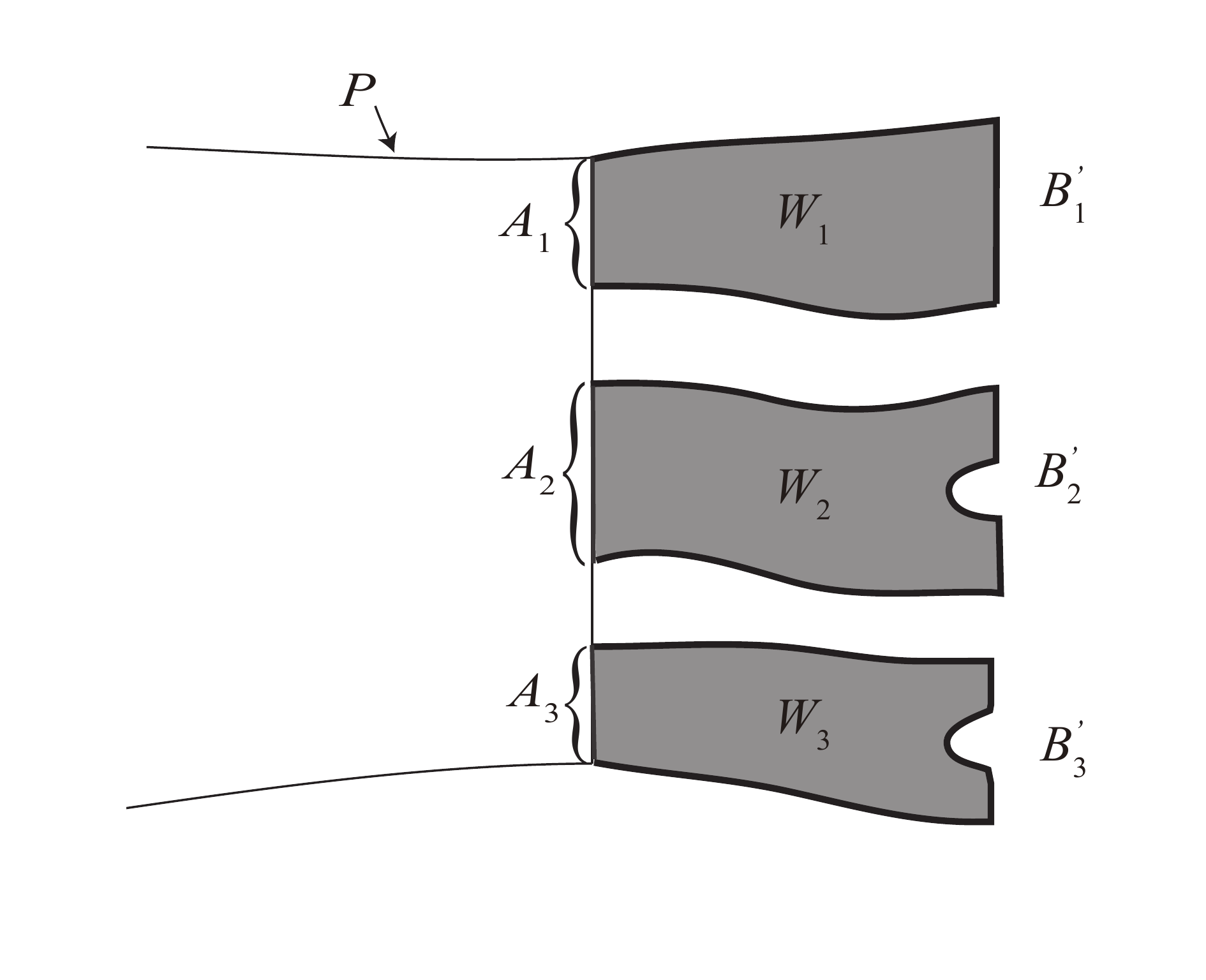}
       \vspace{-1em}
       \caption{$A_i$ is the arc bounded by curly brace and $B_i'$ is thickened black arc. The union of thin arcs and $A_i$'s is $P$.}
        \label{perfect_kernel}
\end{figure}
By Lemma \ref{perfect kernel}, the inclusion induced map $\lambda_1: \pi_1(B_1')\twoheadrightarrow \pi_1(W_1)$ is surjective and $\ker \lambda_1$ is perfect.  Let $L$ be a wedge of loops in $B_1'$ which together generate $\ker \lambda_1$ and $Y_1'$ be the space obtained by attaching 2-cells to the interior $B_1'$ along these loops.  Since $A_1 \hookrightarrow W_1$ is a homotopy equivalence, by Lemma $\ref{amalgamation}$,
$$\pi_1(W_1)\cong \pi_1(A_1)\cong \pi_1(Y') \cong \pi_1(B_1')/\ker \lambda_1,$$
where $\ker \lambda_1$ is the normal closure in $\pi_1(B_1')$ of $\lambda_1(\pi_1(L))$.
Note that $A_1 \cap B_1'= \partial A_1 = \partial B_1'$. By Seifert-van Kampen, 
$$\pi_1((Q \backslash B_1') \cup A_1) \cong \pi_1(Q \cup W_1) \cong \pi_1(Q \cup Y_1').$$
 Let $\iota_{1}^{\ast}: \pi_1(B_1') \to \pi_1(Q)$ be the inclusion induced map. 
Then Lemma \ref{amalgamation} implies $\pi_1(Q \cup Y_1') \cong \pi_1(Q)/\operatorname{ncl}$, where $\operatorname{ncl}$ is the normal closure in $\pi_1(Q)$ of 
$\iota_{1}^{\ast}(\ker \lambda_1)$. Hence, $\phi_1: \pi_1(Q) \twoheadrightarrow \pi_1(Q \cup W_1)$ is surjective and has perfect kernel.

Attaching $W_2$ along $B_2'$ in $Q \cup W_1$. Repeat the above argument, one can show that $\phi_2: \pi_1(Q \cup W_1) \twoheadrightarrow \pi_1(Q\cup W_1 \cup W_2)$ is surjective and has perfect kernel. Assume there are $k$ $A_i$'s. By induction, we have the following sequence
\begin{equation}
\pi_1(Q) \xtwoheadrightarrow{\phi_1} \pi_1(Q \cup W_1) \xtwoheadrightarrow{\phi_2}  \cdots 
\xtwoheadrightarrow{\phi_{k}} \pi_1(Q\cup (\cup_{i=1}^{k} W_i))
\end{equation}
Since each $\ker \phi_i$ is perfect, by Lemma \ref{perfect kernel}, the composition $\Phi = \phi_k \circ \cdots \circ \phi_2 \circ \phi_1$ yields a desired surjection $\pi_1(Q) \twoheadrightarrow \pi_1(R) \cong \pi_1(P)$ and $\ker \Phi$ is perfect.
\end{proof}

\section{Proof of Theorem \ref{Th: Characterization Theorem}:
necessity\label{Section: Proof of Characterization Th: necessity}}

The proof of the necessity of Conditions (\ref{condition a}) and (\ref{condition c}) of Theorem \ref{Th: Characterization Theorem} follow readily by definition of pseudo-collar. Thus, it suffices to show that pseudo-collarability implies Condition (\ref{condition b}).

\begin{proof}[Proof of Theorem \ref{Th: Characterization Theorem} (necessity)]
Suppose $M^m$ is pseudo-collarable and $N$ is a homotopy collar. Then it's easy to see that each component $N^j$ of $N$ is a homotopy collar. 
By the definition of pseudo-collarability, we choose a desired cofinal sequence of clean neighborhoods of infinity $\{N_i^l\}_{l=1}^{k_i}$ such that each $N_i^l$ is a homotopy collar contained in $N^j$. Proposition \ref{Prop: finite ends} guarantees that each $N_i^l \backslash \partial M^m$ is 1-ended --- thus, each $N_i^l$ is $\partial_M N_i^l$-connected at infinity. Let $N_{i,i+s}^{l}= N_i^l \cap (\bigsqcup_{1}^{k_{i+s}} N_{i+s}^{t})$ ($s = 1,2,\dots$) is the disjoint union of finitely many components $N_{i+s}^t$ contained in $N_{i}^{l}$. By Proposition \ref{Prop: one-sided h cobordism}, $N^j$ ($=N_1^1$) can be subdivided into relative one-sided $h$-cobordisms.  That is, each $W_i^l = \overline{N_{i}^{l}\backslash N_{i,i+1}^{l}}$.  By definition, we may consider the sequence
\begin{equation}
\label{semistable}
\pi_1(\partial_{M}N_1^1 \cup N_{1,2}^1)\leftarrow \pi_1(\partial_{M}N_1^1 \cup N_{1,3}^1)\leftarrow \pi_1(\partial_{M}N_1^1 \cup N_{1,4}^1)\leftarrow \cdots
\end{equation}
where base rays are suppressed and bonding homomorphisms are compositions of maps induced by inclusions and change-of-basepoint isomorphisms.
Let $\overline{\partial_M N_i^l\backslash \partial_M N_{i,i+1}^l}$ be $D_{i,i+1}^l$ ($i=1,2,3,\dots$) and $D_{i,i+2}^l=D_{i,i+1}^l \cup D_{i+1,i+2}^l$. Consider the following diagram. Each bonding map in the top row is an inclusion.
\[%
\begin{array}
[c]{ccccccc}%
\partial_M N_1^1 \cup N_{1,2}^1  & \hookleftarrow & \partial_M N_1^1 \cup N_{1,3}^1
& \hookleftarrow &  \partial_M N_1^1 \cup N_{1,4}^1 & \hookleftarrow & \cdots\\
\uparrow \text{incl.} &  & \uparrow\text{incl.} &  & \uparrow\text{incl.} &  & \\
D_{1,2}^1 \cup \operatorname{Fr}N_{1,2}^1  &  & D_{1,3}^1\cup \operatorname{Fr}N_{1,3}^1  &
 & D_{1,4}^1 \cup \operatorname{Fr}N_{1,4}^1 & & \cdots
\end{array}
\]
Since each $\operatorname{Fr}N_i^l \hookrightarrow N_i^l$ is a homotopy equivalence, all the vertical maps are homotopy equivalences. By \P 2 in the proof of Proposition \ref{Prop: one-sided h cobordism}, $(W_i^l, \operatorname{Fr}N_i^l, \operatorname{Fr}N_{i,i+1}^{l} \cup \partial_{M}W_{i}^{l})$ is a one-sided $h$-precobordism. Apply Lemma \ref{Lemma: perfect kernel},
$$\pi_1(D_{1,i+2}^1 \cup \operatorname{Fr}N_{1,i+2}^1) \twoheadrightarrow \pi_1(D_{1,i+1}^1\cup \operatorname{Fr}N_{1,i+1}^1)$$ is surjective
and has perfect kernel.
\end{proof}

\section{Proof of Theorem \ref{Th: Characterization Theorem}:
sufficiency\label{Section: Proof of Characerization Theorem: sufficiency}}
We begin the proof of the "sufficiency argument" with three theorems that will be key ingredients in the proof. Each is a straightforward extension of an established result from the literature.

The following theorem is a modest generalization of the Pseudo-collarability Characterization Theorem in \cite{GT06} to some manifolds
with noncompact boundary in the same way the Siebenmann's "Relativized Main Theorem 10.1" provided a mild extension of the Main Theorem
of \cite{Sie65} to some manifolds with noncompact boundary.

\begin{theorem}[Relativized Pseudo-collarability Characterization Theorem]
\label{Relativized theorem}
Suppose $M^m$ ($m\geq 6$) is one-ended and $\partial M^m$ is homeomorphic to the interior of a compact manifold. 
Then $M^m$ is pseudo-collarable iff $M^m$ is
\begin{enumerate}
\item inward tame,
\item $\pi_1(\varepsilon(M^m))$ is perfectly semistable,
\item $\sigma_{\infty}(M^m) = 0$.
\end{enumerate}
\end{theorem}

Quillen's famous "plus construction" \cite{Qui71} or \cite[Section 11.1]{FQ90} provides a partial converse to Lemma \ref{perfect kernel}.

\begin{theorem}[The Relativized Plus Construction]
\label{Plus construction}
Let $B$ be a compact $(n-1)$-manifold ($n\geq 6$) and $h: \pi_1(B) \twoheadrightarrow H$ a surjective homomomorphism onto a finitely presented group such that $\ker(h)$ is perfect. There exists a compact $n$-dimensional nice relative cobordism $(W,A,B)$ such that $\ker(\pi_1(B)\to \pi_1(W))=\ker h$, and $A \hookrightarrow W$ is a simple homotopy equivalence. These properties determine $W$ uniquely up to homeomorphism rel $B$.
\end{theorem}
\begin{remark}
For $n = 5$, the above theorem still holds as long as $H$ is restricted to be "good" (see \cite[Th. 11.1A, P.195]{FQ90}). For $n\geq 6$, the proof is the same as the proof of Th. 11.1A in \cite[P.195]{FQ90} except that 2-spheres on which the 3-handles are attached embedded simply by general position. When $n = 4$, the theorem is false.
\end{remark}

When a nice rel one-sided $h$-cobordism has trivial Whitehead torsion, ie, when the corresponding
homotopy equivalence is simple, we refer to it as a \emph{nice rel plus cobordism}.

\begin{theorem}[Relativized Embedded Plus Construction]
\label{Embedded plus construction}
Let $R$ be a connected manifold of dimension at least $6$; $B$ a compact codimension $0$ submanifold of $\partial R$; and 
$$G \subseteq \ker(\pi_1(B)\to \pi_1(R))$$
a perfect group which is the  closure in $\pi_1(B)$ of a finite set of elements. Then there exists a nice rel plus cobordism $(W,A,B)$ embedded in $R$ which is the identity on $B$ for which $\ker(\pi_1(B)\to \pi_1(W))= G$.
\end{theorem}

\begin{proof}
The proof of Theorem 3.2 in \cite{GT06} will work for our situation with simple replacement of plus construction by the relativized plus construction and duality by generalized Poincar\'{e} duality \cite[Thm.3.35, P. 245]{Hat02} for noncompact manifolds.
\end{proof}

\begin{proof}[Proof of Theorem \ref{Relativized theorem}]
Consider a manifold $M^m$ whose boundary $\partial M^m$ is homeomorphic to the interior of a compact manifold. Choosing a cofinal sequence of clean neighborhoods $\{N_i\}_i^\infty$ of an end, we must assure that $N_i \cap \partial M^m \approx \partial(N_i \cap \partial M^m) \times [0,1)$. 
With this setup, the notions of generalized $k$-neighborhoods in \cite{Gui00} can be directly applied. Otherwise one may re-define generalized $k$-neighborhoods by using frontiers instead of boundaries. For a full understanding, the reader should be familiar with the proof of the Main Existence Theorem \cite{Gui00}. To generalize all the arguments made in \cite{Gui00}, especially Theorem 5, Lemmas 13-15, one need use
frontiers $\operatorname{Fr}$ of generalized $k$-neighborhoods to replace boundaries $\partial$. All handle operations should be performed away from $\partial M^m$. This is doable for nearly the same reasons given by Siebenmann for \cite[Th.10.1]{Sie65}; in particular, all handle moves in
the proof \cite[Th. 1.1]{GT06} can be performed away from $\partial M^m$. More specifically, the above procedure will assure the end has generalized $(n-3)$-neighborhoods $\{U_i\}$. To modify $\{U_i\}$ to generalized $(n-2)$-neighborhoods, one has to replace Theorem 3.2 in \cite[P.554]{GT06} by Theorem \ref{Embedded plus construction}. Then imitate the argument in \cite[P.554-555]{GT06} via replacing $\partial$ by $\operatorname{Fr}$ and keeping the handle decompositions away from $\partial M^m$. 
\end{proof}

The proof of the sufficiency of Theorem \ref{Th: Characterization Theorem} follows readily from the following result.

\begin{proposition} \label{Prop: nice rel cobordisms}
If $M^m$ satisfies Conditions (\ref{condition a}) - (\ref{condition c}) of Theorem \ref{Th: Characterization Theorem} then there exists a clean compact exhaustion $\{C_i\}$ so that, for the corresponding neighborhoods of infinity $\{N_i\}$,
$\operatorname{Fr}N_i \hookrightarrow N_i$ is a homotopy equivalence.
\end{proposition}

\begin{proof}
The proof is a variation on the argument of Proposition 10.2 in \cite{GG18}. By the definition of peripheral perfect semistability at infinity, we can
begin with a clean compact exhaustion $\{C_i\}_{i}^{\infty}$ of $M^m$ and a corresponding sequence of neighborhoods
of infinity $\{N_i\}_{i=1}^{\infty}$, each with a finite set of connected components $\{N_i^j\}_{j=1}^{k_i}$, so that for all $i\geq 1$ and $1 \leq j \leq k_i$,

\begin{enumerate}[i)]
\item $N_i^j$ is inward tame,
\item $N_i^j$ is $(\partial_M N_i^j)$-connected and $(\partial_M N_i^j)$-perfectly-semistable at infinity, and
\item $\sigma_\infty (N_i^j) = 0$.
\end{enumerate}

By Lemmas \ref{Lemma wall obstruction M - A} and \ref{Lemma: relA pro-pi1 versus Q-A pro-pi1}, we have
\begin{enumerate}[i')]
\item $N_i^j\backslash \partial M^m$ is inward tame,
\item $N_i^j\backslash \partial M^m$ is 1-ended and has perfectly semistable fundamental group at infinity, and
\item $\sigma_\infty (N_i^j \backslash \partial M^m) = 0.$
\end{enumerate}
These are precisely the three conditions of Theorem \ref{Relativized theorem}. In addition, $\partial(N_i^j\backslash \partial M^m) = \operatorname{int}(\operatorname{Fr}N_i^j)$, which is an interior of a compact codimension 1 submanifold of $M^m$. 
That means $N_i^j \backslash \partial M^m$ contains a homotopy collar neighborhood of infinity $V_i^j$, i.e., $\partial V_i^j \hookrightarrow V_i^j$ is a homotopy equivalence. Following the proof of Theorem \ref{Relativized theorem}, one can further arrange $\partial N_i^j \backslash \partial M^m$ $(= \operatorname{int}(\operatorname{Fr}N_i^j))$ and $\partial V_i^j$ contain clean compact codimension 0 submanifolds $A_i^j$ and $B_i^j$, respectively, so that $\operatorname{int}(\operatorname{Fr}N_i^j)\backslash \operatorname{int}A_i^j =\partial V_i^j \backslash \operatorname{int}B_i^j \approx \partial A_i^j \times [0,1)$. See Figure \ref{fig 2}.

\begin{figure}[h!]
        \centering
       \includegraphics[ width=10cm, height=6cm]{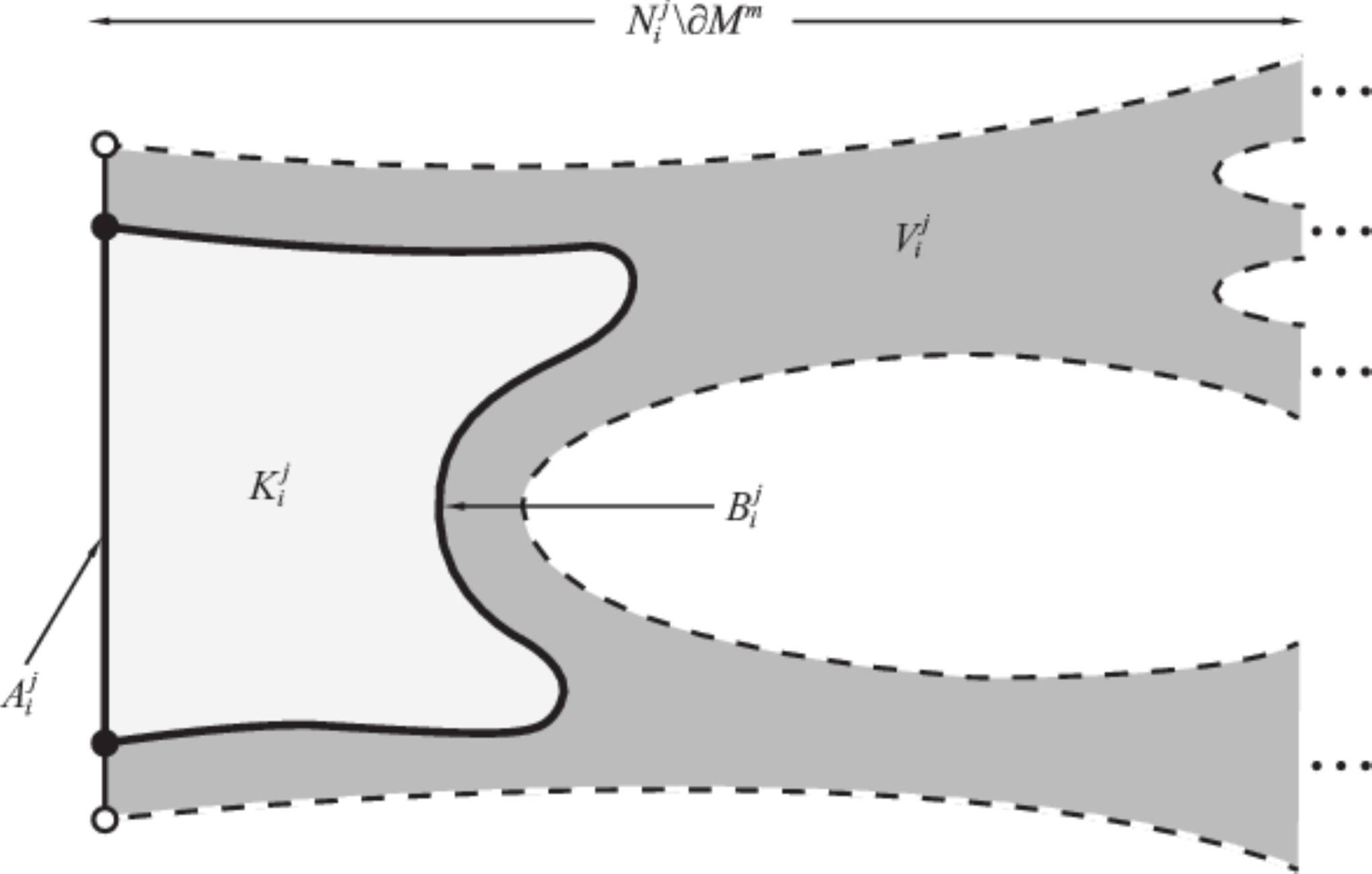}
       \caption{$V_i^j$ is a homotopy collar.}
        \label{fig 2}
\end{figure}

Note that $K_i^j = \overline{N_i^j\backslash V_i^j}$ is a clean codimension 0 submanifold of $M^m$ intersecting $C_i$ in $A_i^j$. To save on notation,
we replace $C_i$ with $C_i \cup (\cup K_i^j)$, which is still a clean compact codimension 0 submanifold of $M^m$, but with the additional property that
\begin{equation}
\operatorname{int}(\operatorname{Fr}N_i)\hookrightarrow N_i \backslash \partial M^m \text{ is  a homotopy equivalence.}
\end{equation}
Since adding $\partial_M N_i$ back in does not affect homotopy types, we have 
\begin{equation}
\operatorname{Fr}N_i \hookrightarrow N_i \text{ is a homotopy equivalence.}
\end{equation}
Having enlarged the $C_i$, if necessary, one can easily retain the property that $C_i \subseteq \operatorname{Int}C_{i+1}$ for all $i$ by passing
to a subsequence. Then $N_i = \overline{M^m\backslash C_i}$ gives a desired nested cofinal sequence of clean neighborhoods of infinity $\{N_i\}$
with the property that each inclusion $\operatorname{Fr}N_i \hookrightarrow N_i$ is a homotopy equivalence, i.e., $M^m$ is pseudo-collarable.
\end{proof}

\section{Questions}
The idea of pseudo-collarability is related to a term named $\mathcal{Z}$-compactification. The motivation first came from the modification of manifold completion applied to Hilbert cube manifolds in \cite{CS76}. A compactification
$\widehat{X}=X\sqcup Z$ of a space $X$ is a $\mathcal{Z}$\emph{-compactification }if, for every open set $U\subseteq\widehat{X}$,
$U\backslash Z\hookrightarrow U$ is a homotopy equivalence. This compactification has been proven to be useful in both
geometric group theory and manifold topology, for example, in attacks on the Borel and Novikov Conjectures. A major open problem is a characterization of $\mathcal{Z}$-compactifiable manifolds \cite{CS76} \cite{GT03} \cite{GG18}.
\begin{question}\label{big question}
Are Conditions (\ref{Char1}), (\ref{Char3}) and (\ref{Char4}) of Theorem \ref{Th: Completion Theorem} sufficient for manifolds to be $\mathcal{Z}$-compactifiable?
\end{question}
Although it's still not well-understood whether these conditions are sufficient, in \cite{GG18}, Guilbault and the author provided a best possible result.

\begin{theorem}
An $m$-manifold $M^m$ ($m\geq 5$) satisfies Conditions (\ref{Char1}), (\ref{Char3}) and (\ref{Char4}) of Theorem \ref{Th: Completion Theorem}, if and only if $M^m \times [0,1]$ admits a $\mathcal{Z}$-compactification.
\end{theorem}

\begin{remark}
Conditions (\ref{Char1}), (\ref{Char3}) and (\ref{Char4}) are precisely the conditions that characterize $\mathcal{Z}$-compactifiable Hilbert cube manifolds \cite{CS76}. The early version of Question \ref{big question} was posed more generally in \cite{CS76} for locally ANR's, but in \cite{Gui01} a counterexample was constructed. 
\end{remark}

Obviously, completable manifolds are both pseudo-collarable and $\mathcal{Z}$-compactifiable. Despite the fact that many manifolds such as Davis's manifolds are both pseudo-collarable and $\mathcal{Z}$-compactifiable but not completable, the relationship between pseudo-collarable manifolds and $\mathcal{Z}$-compactifiable manifolds are not well-understood. There are several interesting questions around such topic.

\begin{question}
Are pseudo-collarability and Condition (\ref{Char4}) of Theorem \ref{Th: Completion Theorem} sufficient for manifolds to be $\mathcal{Z}$-compactifiable?
\end{question}

\begin{question}\label{Question: Z-compact implies pseudo}
Are $\mathcal{Z}$-compactifiable manifolds pseudo-collarable?
\end{question}
We suspect the answer to Question \ref{Question: Z-compact implies pseudo} is negative. Crossing manifolds constructed in \cite{KM62}, \cite{Ste77}
and \cite{Gu18} with half-open interval $[0,1)$ might be potential counterexamples. However, the biggest obstacle is closely related to the following question in knot theory.

\begin{question}
Let $K$ be a trefoil knot and $\operatorname{WD}(K)$ be a twisted Whitehead double of $K$. Is the knot group of $\operatorname{WD}(K)$ hypoabelian?
\end{question} 
\begin{definition}
A group $G$ is said to \emph{hypoabelian} if the following equivalent conditions are satisfied:
\begin{enumerate}
\item $G$ contains no nontrivial perfect subgroup
\item the transfinite derived series terminates at the identity. (Note that this is the transfinite derived series, where the successor of a given subgroup is its commutator subgroup and subgroups at limit ordinals are given by intersecting all previous subgroups.)
\end{enumerate}
\end{definition}

Question \ref{Question: Z-compact implies pseudo} is related to the following open question posed in \cite{GT03} 

\begin{question}
Can a $\mathcal{Z}$-compactifiable open $n$-manifold fail to be pseudo-collarable?
\end{question}

\section*{Acknowledgements}
The work presented here is part of the author's PhD dissertation at University of Wisconsin-Milwaukee. I would like to express my sincere gratitude to my thesis advisor, Craig Guilbault, for his guidance, enthusiasm and encouragement in course of this work.

\end{document}